\documentclass[a4paper, 12pt]{article}

\usepackage[small]{titlesec}
\usepackage{tikz-cd}
\usepackage{booktabs}
\usepackage[T1]{fontenc}
\usepackage[utf8]{inputenc}
\usepackage[english]{babel}
\usepackage[a4paper,top=4cm,bottom=3cm,left=2.5cm,right=2.5cm]{geometry}
\usepackage{amsfonts}
\usepackage{mathtools}
\usepackage{amsmath}

\usepackage{lmodern}
\usepackage{empheq}
\usepackage{amsthm}
\usepackage{amssymb}
\usepackage{booktabs}
\usepackage{caption}
\usepackage{siunitx}
\usepackage{float}
\usepackage{amssymb}
\usepackage{graphicx}
\usepackage{stmaryrd}
\usepackage{color}
\usepackage{pifont}%
\usepackage{enumerate}
\usepackage{enumitem}
\setlist[description]{leftmargin=\parindent,labelindent=\parindent}
\usepackage{listings} 
\usepackage{amsmath}
\usepackage{amscd}
\usepackage{braket}
\usepackage{physics}
\usepackage{tikz}
\usepackage{calc}
\usepackage[hidelinks]{hyperref}

\newcommand{\dom}{\mathrm{dom}}

\newcommand{\R}{\mathbb{R}}
\newcommand{\sph}{\mathbb{S}}

\newtheorem{theorem}{Theorem}[section]

\newcommand{\supp}{\mathrm{supp}\ }

\newcommand\blfootnote[1]{%
	\begingroup
	\renewcommand\thefootnote{}\footnote{#1}%
	\addtocounter{footnote}{-1}%
	\endgroup
}
\title{Entire Monge-Ampère equations and weighted Minkowski problems }
\author{Jacopo Ulivelli}
\date{}

\begin{document}
	\maketitle
	
	\begin{abstract}
		In this short note, we prove the existence of solutions to a Monge-Ampère equation of entire type derived by a weighted version of the classical Minkowski problem.
		\blfootnote{
			MSC 2020 Classification: 52A20 ,35J96, 26B25.\\
			Keywords: Convex function, Monge-Ampère equation, Minkowski problem.}
	\end{abstract}
	
	\section{Introduction}
	Consider a Monge-Ampère equation of the form
	\begin{equation}\label{MAequation}
		c_u\phi(Du(x),u^*(Du(x)))\det D^2 u(x)=f(x), \quad x \in \R^n.
	\end{equation}
	Here, $c_u>0$ is a constant depending on the solution, $u^*$ is the \textit{Fenchel-Legendre transform} of $u$ \[ u^*(x)=\sup_{y \in \R^n}\{ x\cdot y-u(y)\},\] $f \in L^1(\R^n)$ is non-negative, and $\phi:\R^{n+1}\to \R$ is a non-negative, continuous, and even function. With more generality, we can work with a Borel measure $\rho$ on the right-hand side of \eqref{MAequation}. To do so, we focus on weak solutions of \eqref{MAequation}. That is, convex functions $u$ such that 
	\begin{equation}\label{eq:weak_r_curv}
		\rho(B)=\int_{\partial u(B)}c_u\phi((x,u^*(x)))\,dx,
	\end{equation}
	where $\partial u$ is the \textit{subgradient} of $u$. The reader can find details about the theory of convex functions in \cite{Rockafellar1970}. Clearly, \eqref{MAequation} is recovered when $\rho$ has continuous density $f$ with respect to the Lebesgue measure. Following the notation of Bakelman \cite{Bakelman}, we denote the right-hand side of \eqref{eq:weak_r_curv} as $\omega(B,u,c_u\phi)$. By the same procedure as in \cite[Section 9.6]{Bakelman} it is easy to check that this is a Borel measure. There, \eqref{MAequation} was studied when $\phi$ depends only on $Du$. Other instances of the problem have been studied by Chow and Wang \cite{Wang_Minkowski, EntireSolutions} and Bielawski \cite{Bielawski}.
	
	For our strategy we make use of a recent result from Kryvonos and Langharst \cite{DylLyu} (see Theorem \ref{thm:general_Minkowski} below) on the existence of convex compact sets with prescribed weighted surface-area measure. The method we propose allows a quick translation of results on convex bodies to results on convex functions, following a procedure similar to the one employed by Knoerr and the author in \cite{FromConvToFunc2023}. Moreover, the specific structure of \eqref{MAequation} allows weaker assumptions than the ones usually required in Monge-Ampère equations depending on the solution and its gradient (see, for example, \cite[Section 18]{Bakelman}).
	
	Let $\mu$ be a Borel measure on $\R^{n+1}$ with continuous density $\phi$ such that 
	\begin{equation}\label{eq:cond_weigh_mink}
		\lim_{r\to \infty} \frac{\mu(rB^2_{n+1})^{\frac{\beta}{n+1}}}{r}=0 \text{ and } \lim_{r\to 0^+} \frac{\mu(rB^2_{n+1})^{\frac{\beta}{n+1}}}{r}=+\infty
	\end{equation}
	for some $\beta>0$, where $B^2_{n+1}$ is the Euclidean unit ball in $\R^{n+1}$. These hypotheses are precisely those prescribed by Kryvonos and Langharst in \cite{DylLyu}. Our main result reads as follows.
	\begin{theorem}\label{thm:main}
		Consider a Borel measure $\rho$ on $\R^n$ that is not concentrated on an affine hyperplane. Consider, moreover, a continuous even function $\phi:\R^{n+1}\to [0,\infty)$. Then, if $\rho$ has finite first moment, i.e. \[ \int_{\R^n}|x|\,d\rho(x)<+\infty,\] and the measure $\mu$ with density $\phi$ with respect to the Lebesgue measure satisfies \eqref{eq:cond_weigh_mink}, there exist a constant $c_u>0$ and a convex function $u$ such that
		\begin{equation}\label{eq:weak_Ma}
			\omega(B, u,c_u\phi)=\rho(B)
		\end{equation}
		for every Borel set $B \subset \R^n$.
	\end{theorem}
	We remark that the constant $c_u$ is not necessarily unique, nor are the solutions of \eqref{eq:weak_Ma}. For a further discussion see Section \ref{remarks}. The role of such constant is due to the high non-homogeneity of the problem and is clarified later with Theorem \ref{thm:general_Minkowski}. Notice, moreover, that $\phi$ is required to be even as a function on $\R^{n+1}$, which is weaker than asking for symmetry on the first $n$ components. In particular, the solutions found in Theorem \ref{thm:main} are not necessarily symmetric.
	
	Furthermore, we provide the following expected regularity result.
	\begin{theorem}\label{thm:regularity}
		In the hypotheses of Theorem \ref{thm:main}, suppose moreover that $\rho$ has continuous density $f$ with respect to the Lebesgue measure. If $f$ and $\phi$ are such that there exists $c>0$ such that $f,\phi>c$ and of class $C^{k,\alpha}$ for some $k \geq 0$ and $\alpha>0$, then any convex weak solution of \eqref{MAequation} is of class $C^{k+2,\alpha}$.
	\end{theorem}
	\section{Proofs}
	
	In the following we consider the $(n+1)$-dimensional Euclidean space $\R^{n+1}$ with the Euclidean norm $|\cdot|$ and the usual scalar product $x \cdot y$ for $x,y \in \R^{n+1}$. Let $K$ be a compact convex subset of $\R^{n+1}$. For background material on convex geometry, see Schneider's monograph \cite{schneider_2013}. For every $\xi$ in the unit sphere $\sph^{n}$ of $\R^{n+1}$ consider the set $\tau_K(\xi)$ of all the points on the boundary of $K$, denoted $\partial K$, such that there is a tangent hyperplane at $x$ with outer normal $\xi$ for every $x \in \tau_K(\xi)$. The map $\tau_K$ is known as \textit{reverse spherical image}. If $\mu$ is a Borel measure on $\R^{n+1}$ with continuous density $\phi$ with respect to the Lebesgue measure, we define on $\sph^{n+1}$ the $\mu$-surface-area measure of $K$ by 
	\begin{equation}\label{eq:weighted_measures}
		S_K^\mu(B)\coloneq\int_{\tau_K(B)}\phi(y)\, d \mathcal{H}^n(y)
	\end{equation}
	for every Borel set $B \subset \sph^n$, where $\mathcal{H}^n$ is the $n$-dimensional Hausdorff measure restricted to $\partial K$. When $\mu$ is the Lebesgue measure, the apex $\mu$ is omitted and $S_K$ is the classical surface-area measure (see \cite[Section 4]{schneider_2013}). 
	
	Existence theorems for compact convex sets with given $\mu$-surface-area measure have been studied, for example, by Livschytz \cite{LivMi} and Kryvonos and Langharst \cite{DylLyu}. The particular case where $\mu$ is the Gaussian measure has been covered in depth by Huang, Xi, and Zhao \cite{Zhao_Gauss}. Recovering a convex set by its surface-area measure and generalizations of this procedure are known as Minkowski problems. We recommend, for example, \cite[Section 8]{schneider_2013} for an introduction.
	
	Our main instrument is the following result.\vspace{-0.9em}
	\begin{theorem}{\cite[Theorem 1.2]{DylLyu}}\label{thm:general_Minkowski}
		Let $\mu$ be an even Borel measure on $\R^{n+1}$ satisfying \eqref{eq:cond_weigh_mink}.
		Suppose $\rho$ is a finite, even Borel measure on $\sph^n$ that is not concentrated in any great subsphere. Then, there exists a centrally symmetric convex compact set $K \subset \R^{n+1}$ such that
		\[ d\rho(\xi)=c_{\mu,K}dS^\mu_K(\xi), \quad c_{\mu,K}\coloneq \mu(K)^{\frac{\beta}{n+1}-1}.\]
	\end{theorem}\noindent
	The constant $c_{\mu,K}$ appearing in the statement of Theorem \ref{thm:general_Minkowski} plays the same role as $c_u$, the one appearing in Theorem \ref{thm:main}. Terms of this kind are often involved in non-homogeneous Minkowski problems. However, if the measure $\mu$ is homogeneous of some positive degree, the constant can be chosen equal to one. See \cite{LivMi} and the discussion in \cite{DylLyu} after Theorem 1.2 for more details on this topic.\medskip
	
	We are now ready to prove our main result.\vspace{-0.9em} 
	\begin{proof}[Proof of Theorem \ref{thm:main}] For a fixed $v \in \R^{n+1}$, we consider $\R^n \subset \R^{n+1}$ as the hyperplane orthogonal to $v$, which will be the domain of \eqref{MAequation}. Once we have a measure $\rho$ on $\R^n$, we can lift it to $\sph^n_{-}\coloneq \{\xi \in \sph^n: \xi \cdot v<0 \}$ through the diffeomorphism 
		\begin{align*}
			L\colon \R^n &\to \sph^n_{-}\\
			x &\mapsto \frac{(x,-1)}{\sqrt{1+|x|^2}}.
		\end{align*}
		Notice that $L$ brings affine hyperplanes to the intersection of a great subsphere with $\sph^n_{-}$.
		
		Consider a finite Borel measure $\rho$ on $\R^n$ with finite first moment, and define the measure \[ \rho'(B)=\int_{B}\sqrt{1+|x|^2}\,d\rho(x)\] for every Borel set $B \subset \R^n$. Notice that $\supp(\rho)=\supp(\rho')$. Now, we lift $\rho'$ through $L$ as a measure on $\sph^n_{-}$ considering its push-forward $(L)_{\sharp}\rho'$ which, for every Borel set $\omega \subset \sph^n_{-}$, is defined as \[(L)_{\sharp}\rho'(\omega)=\rho'(L^{-1}(\omega)).\] Then, we extend it trivially on the rest of $\sph^n$ and consider the symmetrized measure $\bar{\rho}=(L)_{\sharp}\rho'+(L)_{\sharp}\rho'\circ R$, where $R(\xi)=-\xi$. It is easy to check that it is not concentrated on a great subsphere and is symmetric. By Theorem \ref{thm:general_Minkowski} there exists a centrally symmetric convex compact set $K\subset \R^{n+1}$ and a positive constant $c_{\mu,K}$ such that $\bar{\rho}=c_{\mu,K}S_K^\mu$. This body is the starting point to construct our solution.
		
		The rest of the proof will follow the diagram below.
		\begin{equation*}\label{diagram}
			\begin{tikzcd} \partial {K}_{-}  \arrow[swap]{d}{\pi} & \arrow{l}{\tau_K} \mathbb{S}^n_{-} \\ \dom(w) \arrow{r}{\partial w} & \R^n \ar{u}{L} \end{tikzcd} 
		\end{equation*}
		Here, $\pi: \R^{n+1}\to \R^n$ is the orthogonal projection along $v$, $\partial K_{-}$ is the closure of $\tau_K(\sph^n_{-})$, and $w$ is the convex function on $\R^n$ defined as $w(x)\coloneq \inf\{t: x+tv \in K\}$. We claim that $u=w^*$ satisfies \eqref{eq:weak_Ma}. Notice that since $\dom (w)\coloneq \{p \in \R^n \colon w(p)<+\infty \}$ is compact, $u(x)<+\infty$ for every $x \in \R^n$ and thus $\partial u(x) \neq \emptyset$ for every $x \in \R^n$. Since $\partial u(x) \in \dom(w)$ for every $x \in \R^n$, we have that $u$ is Lipschitz. Moreover, notice that $K$ is centrally symmetric with non-empty interior. In particular, $\dom(w)$ contains an open neighborhood of the origin, and thus $u$ is coercive. 
		
		First, consider $x \in \R^n$. By definition, $\partial u(x)=\{p \in \R^n\colon x \in \partial w(p)\}$. Then, by construction $x \in \partial w(p)$ if and only if \[L(x)=\frac{(x,-1)}{\sqrt{1+|x|^2}} \in \tau_K^{-1}\circ \pi^{-1}(p).\] In particular, even if $\partial u(x)$ is not a singleton, it is brought back to a singleton by $\tau_K^{-1}$ (which is just the Gauss map) and thus, for every $x \in \R^n$ \[x=L^{-1}\circ \tau_K^{-1} \circ \pi^{-1}\circ \partial u(x).\]
		
		To conclude, we now check that $u$ is a weak solution of \eqref{eq:weak_Ma}, where $c_u=c_{\mu,K}$. Consider a Borel set $B \subset \R^n$. We have the changes of variables
		\begin{align*}
			\omega(B,u,c_u\phi)&=\int_{\partial u(B)}c_u\phi((x,w(x)))\, dx=\int_{\pi^{-1}\circ \partial u(B)} c_u\frac{\phi(y)}{\sqrt{1+|Dw(\pi(y))|^2}}\, d\mathcal{H}^n(y) \\ &=\int_{\tau_K^{-1}\circ \pi^{-1}\circ \partial u(B)}|\xi \cdot v| c_{\mu,K }\, dS_K^\mu(\xi)=\int_{L^{-1}\circ \tau_K^{-1}\circ \pi^{-1}\circ \partial u(B)} \frac{1}{\sqrt{1+|z|^2}} \, d\rho'(z)\\ &=\int_{L^{-1}\circ \tau_K^{-1}\circ \pi^{-1}\circ \partial u(B)} \, d\rho(z)=\rho(B),
		\end{align*}
		concluding the proof. Here, $1/\sqrt{1+|Dw(\pi(y))|^2}$ is the approximate Jacobian of $\pi$ and $Dw$ is defined almost everywhere on $\dom(w)$ since $w$ is convex. Moreover, we used the fact that the unit normal vector at $(x,w(x))$ is \[\frac{(D w(x),-1)}{\sqrt{1+|D w(x)|^2}}\] whenever $D w$ is defined, and thus, for $z \in \R^n$ and $x \in \partial u(z)$, \[\frac{1}{\sqrt{1+|D w(\pi((x,w(x))))|^2}}=|L(z)\cdot v|=\frac{1}{\sqrt{1+|z|^2}}.\]
	\end{proof}
	
	Suppose now that $\rho$ has continuous density $f$, and that there exists $c>0$ such that $c<f,\phi$. The proof of Theorem \ref{thm:regularity} can be considered classic and follows, for example, the same steps of \cite[Theorem 0.2]{Bielawski}. We include it for the convenience of the reader. For a lighter notation, since we suppose that a convex solution exists, the constant $c_u$ is absorbed in $\phi$.
	\begin{proof}[Proof of Theorem \ref{thm:regularity}]
		As remarked in the previous proof, $u$ is finite and coercive. Therefore, the level sets $\Omega_t\coloneq\{x \in \R^n \colon u(x)<t\}$ are compact for every $t\in \R$. Moreover since $u$ is Lipschitz and $f,\phi$ are strictly positive and continuous, by \eqref{eq:weak_Ma} there exist $0<c_1<c_2$ such that 
		\begin{equation}\label{eq:Ma_ineq}
			c_1 |B| \leq |\partial u(B)|\leq c_2 |B|
		\end{equation}
		for every Borel set $B \subset \Omega_t$ once $t$ is fixed.
		
		Now, thanks to \cite[Corollary 2]{Caffa_adv}, $u$ is strictly convex in $\Omega_t$, and since $t$ is arbitrary, $u$ is strictly convex everywhere. In \cite{Caffa_comm} it was proved that if $u$ satisfies \eqref{eq:Ma_ineq} and is strictly convex, then $u \in C^{1,\beta}(\Omega_t)$, where $\beta$ depends on $t$. Suppose now that $\phi,f$ are of class $C^{0,\alpha}$. First, rewrite \eqref{MAequation} as \[\det D^2u=\frac{f}{\phi((Du,u^*(Du)))}=g,\] where $g$ is locally of class $C^{0,\gamma}$ for some $\gamma>0$ since both $Du$ and $u^*$ are H\"older continuous. By \cite[Theorem 1.2]{Caff_ann} applied to $u-t$ on $\Omega_t$, we infer $u \in C^{2,\gamma}(\Omega_t)$, and therefore $u\in C^2(\R^n)$. Thus, $g$ is $C^{0,\alpha}$ everywhere, and a reproduction of the argument above proves that $u$ is $C^{2,\alpha}(\R^n)$, settling the case $k=0$. The extension to $k\geq 1$ is achieved through a classic induction argument.
	\end{proof}
	
	\section{Final Remarks}\label{remarks}
	
	The proof of Theorem \ref{thm:regularity} shows that if $\rho$ has continuous density $f$ and both $f$ and $\phi$ are strictly positive, then a solution of \eqref{MAequation} is differentiable, and thus it can be rewritten as \[\phi(Du(x),x\cdot Du(x)-u(x))\det D^2 u(x)=f(x), \] weakening considerably the assumptions that usually are required for the existence of solutions to problems depending explicitly on $u$ (see, for example, \cite[(18.4),(18.5)]{Bakelman}). In particular, no regularity or monotonicity is required for $\phi$.
	
	Notice, moreover, that Theorem \ref{thm:regularity} provides a regularity result for Theorem \ref{thm:general_Minkowski} in case the measure on the sphere and the ambient measure have H\"older-continuous densities.
	
	Concerning the non-uniqueness of the solution and the constant $c_u$, we now adapt some remarks by Huang, Xi, and Zhao \cite{Zhao_Gauss}, who studied the instance of the Gaussian measure $\gamma$, that is, $\phi((x_1,\dots,x_{n+1}))=e^{-\sum_{i=1}^{n+1}x_i^{2}}/\sqrt{2\pi}^{n+1}$. A first source of non-uniqueness concerns the exponent $\beta$ in the condition \eqref{eq:cond_weigh_mink}. 
	\begin{theorem}{\cite[Theorem 1.2]{Zhao_Gauss}}\label{thm:Zhao_Gauss}
		Suppose $\rho$ is a finite even Borel measure on $\sph^{n}$ not concentrated in any closed
		hemisphere. Then for each $0<\beta<\frac{1}{n+1}$, there exists a centrally symmetric convex compact set $K \subset \R^{n+1}$
		that
		\[d\rho(\xi)=c_{\gamma,K}dS_{K}^{\gamma}(\xi)\]	
		where 
		\[c_{\gamma,K}=\gamma(K)^{\frac{\beta}{n+1}-1}.\]
	\end{theorem}\noindent
	Thus, for every admissible $\beta$ one can find a potentially different solution. In Theorem 1.4 of the same work, uniqueness is retrieved under further conditions. 
	
	To give a specific example in our setting, for $a>0$ consider in \eqref{MAequation} the particular $f(x)=\frac{a}{\sqrt{1+|x|^2}}$. Calculations show that this is equivalent to prescribing a constant measure in Theorem \ref{thm:Zhao_Gauss}. Then, as noted in \cite{Zhao_Gauss}, since the Gaussian surface area measure of a sphere $B_r \subset \R^{n+1}$ of radius $r>0$  is $dS^{\gamma}_{B_r}(\xi)=\frac{e^{-\frac{r^2}{2}}}{\sqrt{2\pi}^{n+1}}r^n$, the equation \[a=dS^{\gamma}_{B_r}(\xi)\] has two different solutions for $a$ sufficiently small. Denote the radii of these solutions as $r_1$ and $r_2$ respectively. Then, we have that $u_i(x)\coloneq r_i \sqrt{1+|x|^2}$ are both solutions in \eqref{MAequation} with $c_u=a^{-1}$.
	
	We conclude by noting that a further source of non-uniqueness is hidden in an inheritance vice of the proof. Indeed, when extending $(L)_{\sharp}\rho'$ to the whole sphere, we could have added another even measure concentrated in $\sph^n \cap v^{\perp}$, which would have given a different body as result. A possible fix for this issue would be prescribing an asymptotic cone, which is the usual practice for this kind of problem (see, for example, \cite{Bakelman}, or the more recent \cite{Schneider_cones}). We are not aware of how to implement this method in our setting but, if this issue were fixed and the body $K$ found in Theorem \ref{thm:main} was uniformly convex, under suitable regularity (granted by Theorem \ref{thm:regularity}) uniqueness would be achieved following \cite[Section 17.8]{Gilbarg_Trudinger}.
	

\bibliography{ref} 
\bibliographystyle{siam}

\bigskip

\parbox[t]{8.5cm}{
	Jacopo Ulivelli\\
	Dipartimento di Matematica\\
	Sapienza, University of Rome\\
	Piazzale Aldo Moro 5\\
	00185 Rome, Italy\\
	e-mail: jacopo.ulivelli@uniroma1.it}
\end{document}